\documentclass[a4paper,12pt,twoside]{article}
\usepackage{amsfonts}
\usepackage{rotating}
\usepackage{chapterbib}
\usepackage{graphicx, color}
\usepackage[latin1]{inputenc}
\usepackage[english]{babel}
\usepackage{amssymb}
\usepackage{amsmath, eurosym, layout, hyperref} 
\usepackage{times}
\usepackage{amsthm}
\usepackage{makeidx}
\makeindex
\usepackage{linegoal}

\newcommand{\IR}{\mathbb{R}}
\newcommand{\IZ}{\mathbb{Z}}
\newcommand{\IN}{\mathbb{N}}

\newcommand{\IE}{\mathbb{E}}
\newcommand{\IP}{\mathbb{P}}

\newcommand{\sD}{\mathcal{D}}
\newcommand{\sA}{\mathcal{A}}

\newcommand{\sN}{\mathcal{N}}
\newcommand{\sZ}{\mathcal{Z}}
\newcommand{\sT}{\mathcal{T}}

\newcommand{\ub}[1][i]{\underline{\boldsymbol{#1}}}

\newcommand{\nt}{\notag\\}

\newcommand{\convw}{\Longrightarrow}
\newcommand{\convd}{\stackrel{\sD}{\Longrightarrow}}
\newcommand{\goto}{\longrightarrow}

\newtheorem{theorem}{Theorem}
\newtheorem{proposition}[theorem]{Proposition}
\newtheorem{lemma}[theorem]{Lemma}

\theoremstyle{definition}
\newtheorem{remark}[theorem]{Remark}

\newtheorem{definition}[theorem]{Definition}

\newcommand{\bew}{\noindent\textbf{Proof:}\quad}
\newcommand{\ebew}{\hfill\qed\\}

\newsavebox{\prN}

\renewenvironment{proof}{\bew}{\ebew}

\newenvironment{balign*}[1][12pt]{\setlength{\jot}{#1}\nonumber\align}{\endalign}

\renewcommand{\epsilon}{\varepsilon}

\newcommand{\dotsfi}{\quad\leaders\hbox to 25pt{\hss\tiny.\hss}\hfill}

\newcommand{\eh}[1][2]{\frac{1}{#1}}

\title{The Curie-Weiss model\\[3mm] - an approach using moments }
\author{Werner Kirsch\footnote{\texttt{werner.kirsch@fernuni-hagen.de}} \\
Fakult\"{a}t f\"{u}r Mathematik und Informatik\\
FernUniversit\"{a}t in Hagen, Germany\\[3mm]}

\begin{document}
\addtolength{\baselineskip}{2pt}
\maketitle

\section{Introduction}
In this paper, we consider one of the easiest models for magnetism, the Curie-Weiss model. In this model the elementary magnets can take values $+1$ (spin up) and $-1 $ (spin down). Each spin interacts with all the other spins with the same strength. This interaction makes it more likely for two spins to have the same value than to assume opposite values.

More precisely, the spins $X_{1}, \ldots, X_{N}$ are $\{-1,+1\}$-valued random variables. As typical in models of
statistical mechanics, the (joint) probability
distribution of the $X_{1},X_{2},\ldots,X_{N} $ is defined via
a function $H:\{-1,+1\}^{N}\goto\IR$, called the \emph{energy} ( or Hamiltonian),
by the expression
\begin{align}\label{eq:Gibbs}
   \IP\big(X_{1}=x_{1},X_{2}=x_{2},\ldots,X_{N}=x_{N}\big)~=~
   Z^{-1}\;e^{-\beta H(x_{1},x_{2},\ldots,x_{N})}\,,
\end{align}
where $Z$ is a normalization constant to make $\IP$ a probability measure, i.~e.
\begin{align}
   Z~=~\sum_{(x_{1},x_{2},\ldots,x_{N})\in\{-1,+1\}^{N}}
   \;e^{-\beta H(x_{1},x_{2},\ldots,x_{N})}\,.
\end{align}

The parameter $\beta\geq 0 $ plays the role of an inverse `temperature' $T$, $\beta=\eh[T] $. If $\beta=0$, which means $T=\infty$, the random variables $X_{1},X_{2},\ldots,X_{N} $ are actually independent. If $\beta>0$ those $X_{1},X_{2},\ldots,X_{N} $
which minimize $H$ have higher probability. In other words: The
system prefers states with low energy. This preference is more and
more enhanced if $\beta $ grows.

The details of the model under consideration are encoded in the
energy function $H$.
As a rule, $H$ is of the form
\begin{align}
   H(x_{1},x_{2},\ldots,x_{N})~=~-\sum_{i,j=1}^{N} J_{i,j} x_{i} x_{j}\,.
\end{align}
If all $J_{i,j}\geq 0$ (and not all $=0$) the minimum of the energy is attained if
the $X_{i}$ are `aligned', i.~e. all $X_{i}=1$ or all $X_{i}=-1$. Thus, those `configurations' with many $X_{i}=1$ (or with many
$X_{i}=-1$) are more likely than those with almost equal number of $+1$ and $-1$.
Such models are called \emph{paramagnetic}.

Presumably, the most famous example is the
energy function of the Ising model. In this model the indices $i$
of the random variables $X_{i}$ come from a finite subset $I$ of
the lattice $\IZ^{d}$ and the coupling constants $J_{i,j}$ are given
by
\begin{align}
   J_{i,j}~=~\left\{
               \begin{array}{ll}
                 1, & \hbox{if $\|i-j\|=1 $;} \\
                 0, & \hbox{otherwise.}
               \end{array}
             \right.
\end{align}
So, in the Ising model only spins which are nearest neighbors
interact with each other.

In this paper we consider the easiest non trivial model of
magnetisms, the Curie-Weiss model. In this system every spin
interact with every other spin, more precisely the spin $X_{i}$
interacts with the average of all spins, namely:
\begin{align}
   H(x_{1},x_{2},\ldots,x_{N})~&=~-\eh\,\sum_{i=1}^{N}x_{i}\cdot \Big(\eh[N]\sum_{j=1}^{N} x_{j}\Big)\notag\\
&=~-\eh[2N]\sum_{i,j=1}^{N} x_{i} x_{j}\,.
\end{align}

The Curie-Weiss model is interesting since it is accessible
to mathematical method (even not too sophisticated ones) and
yet has a number of
interesting properties physicists expect of a paramagnetic
system, like a phase transition from a purely paramagnetic
phase to a ferromagnetic phase. We will explain this in detail
in the next section.

The \emph{results} we describe and prove below are not new, but rather
well known to the community. However, the \emph{proofs} we present are certainly
not standard, and rather elementary. We use the moment method
to prove both a `law of large numbers' as well as a `central limit
theorem' and a `non-central limit theorem'.

The Curie-Weiss model goes back to Pierre Curie and Pierre Weiss. A systematic mathematical treatment
can be found in \cite{Thompson} and \cite{Ellis}. For the vast literature on the model see the references
in \cite{Ellis} We refer in particular to \cite{EN1} and \cite{EN2}.

Recently there has been increasing interest in proving limit results for Curie-Weiss models with two or more
groups, see \cite{CG, FC, FM, Collet, LS, KT1, KT2, KLSS}.

Besides describing magnetic systems the Curie-Weiss model
is also used to model voting behavior in various election
models, where $X_{i}=1$ (resp. $X_{i}=-1$) means the voter
$i$ votes `yes' (resp. `no') . The basic idea is that voters
tend to vote in a similar way as the other voters in their
constituency (see \cite{Kirsch-HOEC},\cite{CGh},\cite{KL-LSE}, \cite{Kirsch-EMS}).

\section{Definitions and Results}

\begin{definition}
For $N\in\IN$ and $x_{1},x_{2},\ldots,x_{N}\in\{-1,+1\}$ set
\begin{align}
   H_{N}(x_{1},x_{2},\ldots,x_{N})~=~-\eh[2N]\Big(\sum_{i=1}^{N} x_{i}\Big)^{2}\,.
\end{align}
The Curie-Weiss distribution $CW(\beta,N $) is the probability measure $\IP_{\beta,N}$
on $\{-1,+1\}^{N} $ defined by
\begin{align}
   \IP_{\beta,N}(\{ (x_{1},x_{2},\ldots,x_{N}) \})~&=~Z^{-1}\;e^{-\beta H_{N}(x_{1},x_{2},\ldots,x_{N})}\notag\\
&=~Z^{-1}\;e^{-\frac{\beta}{2N} \big(\sum_{i=1}^{N} x_{i}\big)^{2} }\,.
\end{align}
Here, $\beta\geq 0 $ is called the inverse temperature and $Z$ is a normalization constant so that $\IP_{\beta,N}$
is a probability, i.~e.
\begin{align}
   Z~=~\sum_{x_{1},x_{2},\ldots,x_{N}=1}^{N} e^{-\frac{\beta}{2N} \big(\sum_{i=1}^{N} x_{i}\big)^{2} }\,.
\end{align}
By $\IE_{\beta,N}$ we denote the expectation with respect to the probability measure $\IP_{\beta,N}$.

   We say that a sequence $X_{1},X_{2},\ldots,X_{N} $ of
   $\{-1,+1\}$-valued random
   variables on a probability space $(\Omega,\sA,\IP)$ is \emph{Curie-Weiss distributed} with inverse
   temperature $\beta=\eh[T]\geq 0$\   (or $CW(\beta,N$)-distributed) if
\begin{align}
\IP\Big(X_{1}=x_{1},X_{2}=x_{2},\ldots,X_{N}=x_{N}\Big)
   &=~Z^{-1}\;e^{-\frac{\beta}{2N} \big(\sum_{i=1}^{N} x_{i}\big)^{2} }\,.
\end{align}
If $X_{1},X_{2},\ldots,X_{N} $ are $CW(\beta,N $)-distributed we
call
\begin{align}
   S_{N}~:=~\sum_{i=1}^{N} X_{i}
\end{align}
the \emph{total magnetization} of the $X_{1},X_{2},\ldots,X_{N} $.

\end{definition}

\begin{remark}
Suppose $X_{1},X_{2},\ldots,X_{N} $ are $CW(\beta,N $)-distributed random variable.
Since the function $H_{N}$ is invariant under permutation of its arguments, the random
variables $X_{1},X_{2},\ldots,X_{N} $ are exchangeable. In particular, $\IE_{\beta,N }(X_{i}X_{j})=\IE_{\beta,N }(X_{1}X_{2})$
for $i\not=j$.
Moreover, $\IE_{\beta,N}(X_{i})=0$, as $\IP_{\beta,N}(X_{i}=\pm1)=\eh $ and $\IE_{\beta,N}(X_{i}^{2})=1$,
in fact $X_{i}^{2}=1 $.
\end{remark}

In the following we will be concerned with a scheme of random
variables
\begin{align}
  X^{(N)}_{i},\text{\quad with \quad}N=1,2\ldots\text{\quad and\quad} i=1,2,\ldots, N
\end{align}
such that the sequence $X^{(N)}_{1},X^{(N)}_{2},\ldots,X^{(N)}_{N}$
is $CW(\beta,N) $-distributed.

We will be interested in the behavior
of $S^{(N)}_{N}=\sum_{i=1}^{N}X^{(N)}_{i}$.

Note, that the joint distributions of, say, $X^{(N)}_{1}, X^{(N)}_{2}$
and of $X^{(M)}_{1}, X^{(M)}_{2}$ are \emph{different} for $N\not=M$,
since the distribution CW($\beta,N$) depends explicitly on $N$.
In fact, a priori, $X^{N}_{1}$ and $X^{M}_{1}$ are defined on
different probability spaces, so that it doesn't make sense to speak of quantities like $\IE(X^{(N)}_{i}X^{(M)}_{j})$.

With this being said, from now on we drop the superscript $(N)$ and $(M)$ and simply write
\begin{align}
   S_{N}~&=~\sum_{i=1}^{N} X_{i}\\
\intertext{instead of}S^{(N)}_{N}~&=~\sum_{i=1}^{N} X^{(N)}_{i}
\end{align}
whenever it is clear which $N$ is meant. This is an abuse of notation, but a very convenient one.

The first result is a kind of a `law of large numbers'.

\begin{theorem}\label{thm:lln}
   Suppose $X_{1},X_{2},\ldots,X_{N}$ are CW($\beta,N) $-distributed random variables and set $S_{N}=\sum_{i=1}^{N} X_{i}$.
\begin{enumerate}
   \item If $\beta\leq 1$, then
        \begin{align}
           \eh[N]S_{N}~=~\eh[N]\sum_{i=1}^{N} X_{i}~\convd~\delta_{0}\,,
        \end{align}
        where $\convd $ denotes convergence in distribution and
        $\delta_{a}$ is the Dirac measure in $a$.
    \item
       If $\beta>1$ then
        \begin{align}
           \eh[N]S_{N}~=~\eh[N]\sum_{i=1}^{N} X_{i}~\convd~\eh\,\big(\delta_{-m(\beta)}+\delta_{-m(\beta)}\big)\,,
        \end{align}
    where $m(\beta)>0$ is the unique positive solution of the equation
\begin{align}
   x=\tanh(\beta x)\,.
\end{align}
\end{enumerate}
\end{theorem}

Theorem \ref{thm:lln} shows that there is a phase transition
at inverse temperature $\beta=1$, in the sense that the Curie-Weiss system changes its behavior drastically at $\beta=1$. Up to this
point a `law of large numbers' holds: The arithmetic mean of the
spins goes to zero (= the expectation value of $X_{i}$). Above $\beta=1$ the
limiting distribution of the normalized sum of the spins has two
peaks.

We remark that the convergence for $\beta\leq 1 $ can be strengthened to convergence in probability if we realize
all random variables of the same probability space.

Given the law of large numbers in Theorem \ref{thm:lln} one
may hope that there is a central limit theorem for $\beta\leq 1 $.
This is indeed the case for $\beta< 1 $.

\begin{theorem}\label{thm:clt}
   Suppose $X_{1},X_{2},\ldots,X_{N}$ are CW($\beta,N) $-distributed random variables. If $\beta<1 $ then
\begin{align}\label{eq:clt}
   \eh[\sqrt{N}]S_{N}~=~\eh[\sqrt{N}]\sum_{i=1}^{N} X_{i}~~
    \convd~~\sN(0,\eh[1-\beta])\,,
\end{align}
where $\sN(\mu,\sigma^{2}) $ denotes the normal distribution with
mean $\mu $ and variance $\sigma^{2} $.
\end{theorem}

It follows in particular that
\begin{align}
   \IE_{\beta,N}\left(\left(\eh[\sqrt{N}]\sum_{i=1}^{N}X_{i}\right)^{2}\right)~\goto~\eh[1-\beta]
\end{align}
while $\IE_{\beta,N}(X_{i}^{2})=1 $.

The above result suggests that for $\beta=1$ there is no `standard'
central limit theorem. Indeed, we have:

\begin{theorem}\label{thm:nclt}
Suppose $X_{1},X_{2},\ldots,X_{N}$ are $CW(1,N) $-distributed random variables. Then
\begin{align}\label{eq:nclt}
   \eh[N^{3/4}]S_{N}~=~\eh[N^{3/4}]\sum_{i=1}^{N} X_{i}~~
    \convd~~\mu
\end{align}
where $\mu $ is a measure with Lebesgue density $\rho(x)=C\,e^{-\eh[12]x^{4}}$.

\end{theorem}

Since for $\beta>1 $ the expression $\eh[N]S_{N}$ converges to a distribution which is not concentrated in one point, there is no central limit theorem
in the usual sense that for a suitable constant
\begin{align}
   \eh[\sqrt{N}](S_{N}-c)~\convd~\mu\,.
\end{align}

However, there is a `conditional' version of the central limit theorem. For details we refer to \cite{Kirsch-Mom}.

\section{Strategy of the proofs}
To prove convergence in distribution we use the method of moments.

\begin{theorem}[Method of Moments]\label{thm:mom}
 Suppose $\mu_{n}$ and $\mu $ are Borel measure on $\IR $ such that all moments
\begin{align}
 m_{k}(\mu_{n})~:=~\int x^{k}\,d\mu_{n}\qquad
\text{and}\qquad m_{k}(\mu)~:=~\int x^{k}\,d\mu
\end{align}
are finite and such that
\begin{align}
   |m_{k}(\mu)|~\leq~A\,C^{k}\,k!\,.
\end{align}
If for all $k\quad
   m_{k}(\mu_{n})~\goto~m_{k}(\mu)$\quad
then\quad $\mu_{n}~\convw~\mu$.
\end{theorem}
For a proof see e.~g. \cite{Breiman}.

To employ Theorem \ref{thm:mom} we got to estimate expressions
of the form
\begin{align*}
   \IE_{\beta,N}\left(\left(\eh[N^{\alpha}]\sum_{i=1}^{N}X_{i}\right)^{K}\right)
\end{align*}
with $\alpha\in\{ \eh,\frac{3}{4}, 1 \}$.

We have
\begin{align}\label{eq:momK}
   \IE_{\beta,N}\left(\left(\sum_{i=1}^{N}X_{i}\right)^{K}\right)~=~\sum_{x_{1_{1}},x_{i_{2}},\ldots,x_{i_{K}}=1}^{N}\,\IE_{\beta,N}\Big(X_{i_{1}}\cdot X_{i_{2}}\cdot\ldots\cdot X_{i_{K}}\Big)\,.
\end{align}

Note that for pairwise \emph{distinct} $j_{1},\ldots, j_{\ell}$
\begin{align}
   \IE_{\beta,N}\Big(X_{j_{1}}\cdot X_{j_{2}}\cdot\ldots\cdot X_{j_{\ell}}\Big)~=~\IE_{\beta,N}\Big(X_{1}\cdot X_{2}\cdot\ldots\cdot X_{\ell}\Big)\,,
\end{align}
since the measure $\IP_{\beta,N}$ is invariant under permutations of indices (exchangeability).

We observe that $X_{i}^{\ell}=X_{i}$ for odd $\ell$ and $X_{i}^{\ell}=1$ for even $\ell$. Thus

\begin{align}\label{eq:exp}
   \IE_{\beta,N}\Big(X_{i_{1}}\cdot X_{i_{2}}\cdot\ldots\cdot X_{i_{K}}\Big)~=~\IE_{\beta,N}\Big(X_{1}\cdot X_{2}\cdot\ldots\cdot X_{\ell}\Big)\,,
\end{align}
where $\ell\leq K$ is the number of indices $i_{\nu}$ which occur an odd number of times among $i_{1},\ldots,i_{K}$.

In the following section we estimate expectations of the form \eqref{eq:exp}. It turns out that their behavior in $N$ depends strongly on the parameter $\beta $.
In the sections \ref{sec:lln} to \ref{sec:nclt} we use this information to evaluate the moments \eqref{eq:momK} thus proving Theorems \ref{thm:lln}, \ref{thm:clt}
and \ref{thm:nclt}.

\section{Correlations}\label{sec:corr}
In this section we estimate correlations of the form
\begin{align}\label{eq:corr}
   \IE_{\beta,N}\big(X_{1}\cdot X_{2}\cdot\ldots\cdot X_{\ell}\big)\,.
\end{align}

To do so it is convenient to write the probability distribution $\IP_{\beta,N} $
in a form which is more suitable for sending $N$ to infinity. The basic idea, known in
physics as the Hubbart-Stratonovich tranform, is to use the equality
\begin{align}\label{eq:HS}
   e^{a^{2}/2}~=~\eh[\sqrt{2\pi}]\int_{\IR} e^{-\eh x^{2}+ax}\,dx\,,
\end{align}
which is nothing but $\eh[\sqrt{2\pi}]\int e^{\eh(x-a)^{2}}dx =1$.

This observation allows us to write the correlations \eqref{eq:corr} in the following form:
\begin{proposition}\label{pro:HS}
Define $F_{\beta}(t):=\eh[2\beta]t^{2}-\ln\cosh(t)$ and set
\begin{align}
   \sZ_{N}(\ell)~:=~\int_{-\infty}^{+\infty} e^{ N\,F_{\beta}(t)} \tanh^{\ell}(t) dt\,.
\end{align}
Then, for $\ell\leq N$
\begin{align}
    \IE_{\beta,N}\big(X_{1}\cdot X_{2}\cdot\ldots\cdot X_{\ell}\big)
~=~\frac{\sZ_{N}(\ell)}{\sZ_{N}(0)}\,.
\end{align}
\end{proposition}

\begin{proof} By \eqref{eq:HS} we have
   \begin{align}
   & \sT_{N}(\ell)~:=~\eh[2^{N}] \sum_{x_{1},x_{2},\ldots,x_{N}\in\{-1,+1\}} x_{1}x_{2}\ldots x_{\ell}\quad
e^{\frac{\beta}{2N}\Big(\sum_{i=1}^{N}x_{i} \Big)^{2}}\notag\\
=~&\eh[2^{N}\sqrt{2\pi}]\sum_{x_{1},x_{2},\ldots,x_{N}\in\{-1,+1\}} x_{1}x_{2}\ldots x_{\ell}\;
\int_{-\infty}^{\infty}e^{-\eh s^{2}+\frac{\sqrt{\beta}}{\sqrt{N}}(\sum_{i=1}^{N}x_{i})\,s}\;ds\notag\\
~=~&\eh[2^{N}\sqrt{2\pi}]\sum_{x_{1},x_{2},\ldots,x_{N}\in\{-1,+1\}} x_{1}x_{2}\ldots x_{\ell}\;
\int_{-\infty}^{\infty}e^{-N \eh[\beta] t^{2}}\;dt\nt
~~=~&\eh[2^{N}\sqrt{2\pi}]\int_{-\infty}^{\infty}\sum_{x_{2},\ldots,x_{N}\in\{-1,+1\}} \left(\sum_{x_{1}\in\{-1,+1\}}x_{1}e^{x_{1}\,t}\right) x_{2}\ldots x_{\ell}\;
e^{-N \eh[\beta] t^{2}}\;\prod_{i=2}^{N} e^{x_{i}\,t}\;dt\nt
~=~&\eh[2^{N-1}\sqrt{2\pi}]\int_{-\infty}^{\infty}\sum_{x_{2},\ldots,x_{N}\in\{-1,+1\}} \sinh(t)~~x_{2}\ldots x_{\ell}\;
e^{-N \eh[\beta] t^{2}}\;\prod_{i=2}^{N} e^{x_{i}\,t}\;dt\nt
~=~&\eh[2^{N-\ell}\sqrt{2\pi}]\int_{-\infty}^{\infty}\sum_{x_{\ell+1},\ldots,x_{N}\in\{-1,+1\}} \sinh^{\ell}(t)~~
e^{-N \eh[\beta] t^{2}}\;\prod_{i=\ell+1}^{N} e^{x_{i}\,t}\;dt\nt
~=~&\eh[\sqrt{2\pi}]\int_{-\infty}^{\infty}\sinh^{\ell}(s)\,\cosh^{N-\ell}(s)~~
e^{-N \eh[\beta] t^{2}}\;dt\nt
=~&\frac{\sqrt{N}}{\sqrt{2\pi\beta}}\,\int_{-\infty}^{\infty}e^{-N\big(\frac{t^{2}}{2\beta}-\ln\cosh t\big)}\;\tanh^{\ell}(t)\;dt\notag\,.
   \end{align}
Consequently
\begin{align}
   \IE_{\beta,N}\Big(X_{1}X_{2}\ldots X_{\ell}\Big)~=~\frac{\sT_{N}(\ell)}{\sT_{N}(0)}
~=~\frac{\sZ_{N}(\ell)}{\sZ_{N}(0)}\,.
\end{align}
\end{proof}

By symmetry we see that $\sZ_{N}(\ell)=0 $ for odd $\ell$. To estimate $\sZ_{N}(\ell)$
for even $\ell$ we use Laplace's method:

\begin{theorem}[Laplace]\label{pro:Laplace}
Suppose the smooth function $F:\IR\to\IR$ has a unique global minimum at $t_{0}$ with
$F^{(m)}(t_{0})>0$ for an even $m$ and $F^{(r)}(t_{0})=0$ for all $0\leq r<m $, moreover
let $\varphi$ be a bounded continuous function which is continuous at $t_{0}$ with $\varphi(t_{0})\not=0 $.

If
   $\int_{-\infty}^{+\infty} e^{-N F(t)}\, |t^{\ell}|\, dt$
is finite for all $\ell$ and all $N$ large enough , then
\begin{align}\label{eq:Laplace}
   \int_{-\infty}^{+\infty} e^{-N\, F(t)} t^{\ell} \varphi(t) dt
~\underset{N\to\infty}{\approx}~
\Big(\frac{1}{N\,F^{(m)}(0)}\Big)^{\frac{\ell+1}{m}}\,\varphi(0) \int_{-\infty}^{+\infty}e^{-\eh[m!]\,t^{m}}\; t^{\ell}\, dt\,.
\end{align}

\end{theorem}
The Laplace-Theorem in the form we need it here can be deduced from \cite{Olver}. For the reader's
convenience we give a rough sketch of a proof in the Appendix (section \ref{sec:app}).

Propositions \ref{pro:HS} and \ref{pro:Laplace} allow us to compute the asymptotic behaviour
of the correlations \eqref{eq:corr}.

\begin{theorem}
\label{thm:Corr}
Suppose $X_1,X_2, \cdots ,X_{\ell}, X_{\ell+1}, \ldots, X_{N}$ are $CW(\beta,N)$-distributed variables.\\[2mm]
If $\ell$ is even, then as $N \rightarrow \infty$:
\begin{enumerate}
\item If $\beta < 1$, then
\begin{equation}
\IE_{\beta,N} \big(X_1 \cdot X_2 \cdot X_{\ell}\big) ~\approx~ (\ell - 1)!!\; \Big(\frac{\beta}{1 - \beta}\Big)^{\frac{\ell}{2}} \;\frac{1}{N^{\frac{\ell}{2}}}\,.
\label{eq:CorrCW1}
\end{equation}
\item If $\beta = 1$, then 
\begin{equation}\label{eq:CorrCW2}
\IE_{1,N} \big(X_1 \cdot X_2 \cdot X_{\ell}\big) \approx  \frac{1}{N^{\frac{\ell}{4}}}
\frac{\int t^{\ell}\;e^{-\frac{1}{12}t^{4}}\,dt}{\int \;e^{-\frac{1}{12}t^{4}}\,dt}\,.
\end{equation}
\item If $\beta > 1$, then
\begin{equation}\label{eq:CorrCW3}
\IE_{\beta, N} \big(X_1 \cdot X_2 \cdot X_{\ell}\big) \approx\; m(\beta)^{\ell}\,,
\end{equation}
where $t = m(\beta)$ is the strictly positive solution of $\tanh \beta t = t$.
\end{enumerate}
If $\ell$ is odd then $\IE_{\beta,N} \big(X_1 \cdot X_2 \cdot X_{\ell}\big) = 0$ for all $\beta$.

\end{theorem}
\begin{remark}
   Up to the factor $N^{-\ell/2}$ \eqref{eq:CorrCW1} is the $\ell^{\text{th}}$ moment of the normal distribution $\sN(0,\frac{\beta}{1-\beta})$, \eqref{eq:CorrCW2} are the moments
   of a probability measure with density proportional to $e^{-\frac{1}{12}t^{4}}$ up to the factor $N^{-\ell/4}$, and \eqref{eq:CorrCW3} are the moments of the measure
   $\eh\big(\delta_{-m(\beta)}+\delta_{m(\beta)}\big)$.
\end{remark}

\begin{proof}
We compute:
\begin{align}
   F_{\beta}'(t)~&=~\eh[\beta]\; t -\tanh t\,,\\
   F_{\beta}''(t)~&=~\eh[\beta]-\eh[\cosh^{2} t]\,.
\end{align}
   Thus, for $\beta<1$ the function $F_{\beta}$ is strictly convex and has a local minimum at $t=0$.
Consequently, this minimum is global and we can apply Proposition \ref{pro:Laplace} to find
\begin{align}
   \IE_{\beta,N}\Big(X_{1}\cdot \ldots\cdot X_{\ell}\Big)~&=~ \frac{\sZ_{N}(\ell)}{\sZ_{N}(0)}\notag\\
&=~ \int_{-\infty}^{+\infty} e^{ N\,F_{\beta}(t)} \tanh^{\ell}(t) dt \quad
\Big(\int_{-\infty}^{+\infty} e^{ N\,F_{\beta}(t)}  dt\Big)^{-1}\nt
&=~ \int_{-\infty}^{+\infty} e^{ N\,F_{\beta}(t)}\; t^{\ell}\;\frac{\tanh^{\ell}(t)}{t^{\ell}}\; dt \quad
\Big(\int_{-\infty}^{+\infty} e^{ N\,F_{\beta}(t)}  dt\Big)^{-1}\nt
&\approx~\eh[N^{\ell/2}]\big(\frac{\beta}{1-\beta}\big)^{\ell/2}\;\eh[\sqrt{2\pi}]\int t^{\ell}
e^{-t^{2}/2}\,dt\nt
&=~\big(\ell-1\big)!!\; \big(\frac{\beta}{1-\beta}\big)^{\ell/2}\; \eh[N^{\ell/2}]\notag\,.
\end{align}

For $\beta=1 $ we obtain $t=0$ is still the unique solution of $F_{1}'(t)=0$,
$F_{1}^{(2)}(0)=F_{1}^{(3)}(0)=0$ and $F_{1}^{(4)}=2$. Thus, $t=0$ is a global minimum of $F_{1}$
and the above reasoning gives \eqref{eq:CorrCW2}.

For $\beta>1 $ we have $F_{\beta}'(0)=0$ and $F_{\beta}''(0)=\eh[\beta]-1<0$, so $0$ is a local
maximum.

Since $F_{\beta}(t)=F_{\beta}(-t)$ we have for $r$ even:
\begin{align}
   \sZ_{N}(r)~&=~\int_{-\infty}^{0} e^{ N\,F_{\beta}(t)} \tanh^{r}(t) dt~+~\int_{0}^{\infty} e^{ N\,F_{\beta}(t)} \tanh^{r}(t) dt\notag\\
~&=~2\;\int_{0}^{\infty} e^{ N\,F_{\beta}(t)} \tanh^{r}(t) dt\label{eq:int2}\,.
\end{align}
Thus, it suffices to estimate the integrals \eqref{eq:int2} for $r=\ell$ and $r=0$.

Set $f(t)=\eh[\beta]t$ and $g(t)=\tanh(t)$, so $F_{\beta}'(t)=f(t)-g(t)$.

We have $f(0)=g(0)$ and, due to $\beta>1$, $f'(0)<g'(0)$, hence $f(t)<g(t)$ for small $t>0$.
Moreover, $g$ is bounded and strictly concave (for $t>0 $).
Consequently, there is a unique $t_{0}>0$ with $F_{\beta}'(t_{0})=f(t_{0})-g(t_{0})=0$.
We have $g'(t_{0})<f'(t_{0})$ due to the concavity of $g$, hence $F_{\beta}''(t_{0})>0$.

By Proposition \ref{pro:Laplace} we obtain
\begin{align}
   \IE_{\beta,N}\Big(X_{1}\cdot\ldots\cdot X_{\ell}\Big)~\approx~{\big(\frac{t_{0}}{\beta}\big)}^{\ell}=:m(\beta)^{\ell}\,.
\end{align}
We have
\begin{align}
   \tanh\big(\beta m(\beta)\big)~=~\tanh\left(t_{0}\right)=~\beta t_{0}~=~m(\beta)\,.
\end{align}
This proves
\eqref{eq:CorrCW3}.
\end{proof}
\section{Proof of Theorem \ref{thm:lln}}\label{sec:lln}
We estimate
\begin{align}
   &\IE_{\beta,N}\,\Big(\big(\eh[N]\sum_{i=1}^{N}\big)^{K}\Big)\nt
=~&\eh[N^{K}]\IE_{\beta,N}\,
\Big(\sum_{i_{1},i_{2},\ldots,i_{K}=1}^{N}\,X_{i_{1}}\cdot X_{i_{2}}\cdot\ldots\cdot X_{i_{K}}\Big)\,.
\label{eq:bigsum}\end{align}
Evaluating these sums is a combination of bookkeeping and correlation estimates as in \ref{sec:corr}. To do the bookkeeping we define:
\begin{definition}\label{def:comb1} We set
\begin{align}
   W_{K,N}~&:=~\{ \ub=(i_{1},i_{2},\ldots,i_{K})\mid 1\leq i_{j}\leq N \}\\
    W_{K,N}(r)~&:=\{ \ub\in W_{K,N}\mid \text{exactly $r$ different indices occur once in $\ub$} \}
\end{align}
By $w_{K,N}$ and $w_{K,N}(r)$ we denote the number of multiindices in $W_{K,N}$ and $W_{K,N}(r)$ respectively.
\end{definition}
\begin{lemma}\label{lem:comb11}
\begin{align}\label{eq:comb1}
   &w_{K,N}(r)~\leq~K!\,N^{\frac{K+r}{2}}\,,\\
   &w_{K,N}(K)~=~\frac{N!}{(N-K)!}~\approx~N^{K}\label{eq:comb1K}\,.
\end{align}
\end{lemma}

\begin{proof}
   The multiindices in $W_{K,N}(r)$ contain at most $r+\frac{K-r}{2}=\frac{K+r}{2}$ different indices. There are at most $N^{\frac{K+r}{2}} $ ways to choose them
   and at most $K!$ ways to order them.

   For $r=K$ we have $\frac{N!}{(N-K)!}\approx N^{K}$ possibilities to choose an ordered $K$-tuple from $N$ indices (without repetition).
   \end{proof}

We estimate
\begin{align}
   \IE_{\beta,N}\left(\left(\eh[N]\sum_{i=1}^{N}X_{i}\right)^{K}\right)
=~&\eh[N^{K}]\sum_{r=0}^{K-1}
\sum_{\ub\in W_{K,N}(r)} \IE_{\beta,N}\,\Big(X_{i_{1}}\cdot X_{i_{2}}\cdot\ldots\cdot X_{i_{K}}\Big)\nt
~&+~\eh[N^{K}]\sum_{\ub\in W_{K,N}(K)} \IE_{\beta,N}\,\Big(X_{i_{1}}\cdot X_{i_{2}}\cdot\ldots\cdot X_{i_{K}}\Big)\nt
~\approx~~~&\eh[N^{K}]C N^{K-\eh}~+~\IE_{\beta,N}\,\Big(X_{{1}}\cdot X_{{2}}\cdot\ldots\cdot X_{K}\Big)\nt
\approx~~~&\IE_{\beta,N}\,\Big(X_{{1}}\cdot X_{{2}}\cdot\ldots\cdot X_{K}\Big)\,.
\end{align}
The last expression goes to $0$ for $\beta\leq 1 $ by Theorem \ref{thm:Corr}.
For $\beta> 1 $ it converges to $m(\beta)$ for even $K$ and to $0$ for odd $K$.

Together with Theorem \ref{thm:mom} this proves Theorem \ref{thm:lln}.

\section{Proof of Theorem \ref{thm:clt}}\label{sec:clt}
In our proof of Theorem \ref{thm:lln} we realized that only terms with $K$ distinct indices
counted in the limit for \eqref{eq:bigsum}. For the central limit theorem for \emph{independent} random variables
the only important terms are those with all indices occurring exactly twice.

It will turn out that for the Curie-Weiss model with $\beta<1 $ both doubly occurring indices
and those that occur only once play a role in the limit.

To do the bookkeeping we got to refine our definitions in Definition \ref{def:comb1}.
\begin{definition}\label{def:comb2}
   We set
\begin{align}
   W_{K,N}^{0}(r)~&=~\{\ub\in W_{K,N}(r) \mid \text{no index occurs more than twice.} \}\\
   W_{K,N}^{+}(r)~&=~W_{K,N}(r)\;\setminus\;W_{K,N}^{0}(r)\,.
\end{align}
and denote by $w_{K,N}^{0}(r) $ and the $w_{K,N}^{+}(r) $ the cardinality of $W_{K,N}^{0}(r) $
and $W_{K,N}^{+}(r) $ respectively.
\end{definition}

\begin{lemma}\label{lem:comb2}
\begin{align}
   w_{K,N}^{+}(r)~\leq~K!\,N^{\frac{K+r}{2}-\frac{1}{2}}\,.
\end{align}
\end{lemma}

\begin{proof}
   If the $K$-tuple $\ub$ contains $r$ indices with only one occurrence and at least one index with three or more occurrences
   there are at most $r-3$ places left for indices with (exactly) two occurrences. Therefore, a tuple in $w_{K,N_{1}}^{+}(r)$
   contains at most $r   + 1 +\frac{K-r-3}{2}$ \emph{different} indices. Consequently there are at most
   $K!\,N_{1}^{\frac{K+r}{2}-\frac{1}{2}}$ such tuples.
\end{proof}

\begin{lemma}\label{lem:comb4}
\begin{align}
   w_{K,N}^{0}(r)~=~\left\{
                          \begin{array}{ll}
                           \frac{N!}{(N-\frac{K+r}{2})!}\;\frac{K!}{r!\;(\frac{K-r}{2})!\; 2^{\frac{K-r}{2}}}\; , & \hbox{if $K-r$ is even;} \\
                            0, & \hbox{else.}
                          \end{array}
                        \right.
  \end{align}
\end{lemma}
\begin{proof}
   We choose an (ordered) $r$-tuple $\rho$ of $r$ indices to occur once and an ordered  $(K-r)/2$-tuple $\lambda$ of indices to occur twice in $\ub$.
We have
\begin{align*}
   \frac{N!}{(N-\frac{K+r}{2})!}
\end{align*}
ways to do so.

Then we choose the $r$ positions for those indices which occur once. We can do this in
\begin{align*}
   \binom{K}{r}~=~\frac{K!}{r!\;(K-r)!}
\end{align*}
ways. We fill these positions in $\ub$ with $\rho_{1}, \rho_{2},\ldots,\rho_{r}$ starting with the left most open position.

Finally, we distribute the indices $\lambda_{1},\ldots,\lambda_{(K-r)/2}$, twice each. The index $\lambda_{1}$ is put at the left most free place in $\ub$ and in
one of the remaining $K-r-1$ positions, $\lambda_{2}$ is put at the then first free place in $\ub $ and in one of the $K-r-3$ remaining free places and so on.

This gives
\begin{align}
   (K-r-1)!!~=~\frac{(K-r)!}{(\frac{K-r}{2})!\; 2^{\frac{K-r}{2}}}
\end{align}
 possibilities.
\end{proof}

We are now in a position to complete the proof of Theorem \ref{thm:clt}.

We split the sum
\begin{align}
      &\IE_{\beta,N}\left(\left(\eh[\sqrt{N}]\sum_{i=1}^{N}\right)^{K}\right)\nt
=~\eh[N^{K/2}]\;&\IE_{\beta,N}
\Big(\sum_{i_{1},i_{2},\ldots,i_{K}=1}^{N}\,X_{i_{1}}\cdot X_{i_{2}}\cdot\ldots\cdot X_{i_{K}}\Big)\notag
\intertext{into two parts:}
=~\eh[N^{K/2}]\;&\sum_{r=0}^{K}\sum_{\ub\in W_{K,N}^{0}(r)}\IE_{\beta,N}
\Big(X_{i_{1}}\cdot X_{i_{2}}\cdot\ldots\cdot X_{i_{K}}\Big)\label{eq:bigsumw1}\\
~+~\eh[N^{K/2}]\;&\sum_{r=0}^{K}\sum_{\ub\in W_{K,N}^{+}(r)}\IE_{\beta,N}
\Big(X_{i_{1}}\cdot X_{i_{2}}\cdot\ldots\cdot X_{i_{K}}\Big)
\label{eq:bigsumw2}\,.
\end{align}

We estimate \eqref{eq:bigsumw2} first. If $\ub\in W_{K,N}^{+}(r)$ then
\begin{align}
   \IE_{\beta,N}
\Big(X_{i_{1}}\cdot X_{i_{2}}\cdot\ldots\cdot X_{i_{K}}\Big)~=~
\IE_{\beta,N}\Big(X_{1}\cdot X_{2}\cdot\ldots\cdot X_{r}\;\cdot \ldots\cdot X_{r+s} \Big)\label{eq:bsw3}
\end{align}
since $X_{i}^{\ell}=1 $ for even $\ell $ and $X_{i}^{\ell}=X_{i} $ for odd $\ell $. (In \eqref{eq:bsw3} $s$ may be $0$.)

Consequently for $\ub\in W_{K,N}^{+}(r)$ Theorem \ref{thm:Corr} part $3$  gives
\begin{align}
   \IE_{\beta,N}
\Big(X_{i_{1}}\cdot X_{i_{2}}\cdot\ldots\cdot X_{i_{K}}\Big)~\leq~C_{1}\;N^{-r/2}\,.
\end{align}

By Lemma \ref{lem:comb2} we conclude that

\begin{align}
   \eh[N^{K/2}]\;&\sum_{r=0}^{K}\sum_{\ub\in W_{K,N}^{+}(r)}\IE_{\beta,N}
\Big(X_{i_{1}}\cdot X_{i_{2}}\cdot\ldots\cdot X_{i_{K}}\Big)~\leq~C_{2}\;N^{-1/2}\label{eq:wk+}\,.
\end{align}

The remaining, in fact leading, term is
\begin{align}
  &\eh[N^{K/2}]\;\sum_{r=0}^{K}\sum_{\ub\in W_{K,N}^{0}(r)}\IE_{\beta,N}
\Big(X_{i_{1}}\cdot X_{i_{2}}\cdot\ldots\cdot X_{i_{K}}\Big)\notag \\
=~&\eh[N^{K/2}]\;\sum_{r=0}^{K}\sum_{\ub\in W_{K,N}^{0}(r)}\IE_{\beta,N}
\Big(X_{1}\cdot X_{2}\cdot\ldots\cdot X_{r}\Big) \label{eq:clta}\,.
\end{align}
Since $K$ is even and $\IE_{\beta,N}
(X_{1}\cdot X_{2}\cdot\ldots\cdot X_{r})=0 $ for odd $r$ we may set $K=2L$ and write \eqref{eq:clta} as
\begin{align}
   &\eh[N^{L}]\;\sum_{\ell=0}^{L}\sum_{\ub\in W_{2L,N}^{0}(2\ell)}\IE_{\beta,N}
\Big(X_{1}\cdot X_{2}\cdot\ldots\cdot X_{2\ell}\Big)\nt
\approx~&\eh[N^{L}]\;\sum_{\ell=0}^{L}\,\frac{N!}{\big(N-(L+\ell)\big)!}\,\frac{(2L)!}{(2\ell)!\,(L-\ell)!\,2^{L-\ell}}\;(2\ell-1)!!\,\Big(\frac{\beta}{1-\beta}\Big)^{\ell}\,N^{-\ell}\nt
\approx~& \sum_{\ell=0}^{L}\,\frac{(2L)!}{(2\ell)!\,(L-\ell)!\,2^{L-\ell}}\;(2\ell-1)!!\,\Big(\frac{\beta}{1-\beta}\Big)^{\ell}\nt
=~& \frac{(2L)!}{L!\,2^{L}}\;\sum_{\ell=0}^{L}\,\frac{L!}{(L-\ell)!\,\ell!}\,\Big(\frac{\beta}{1-\beta}\Big)^{\ell}\nt
=~&(2L-1)!!\;\Big(\frac{1}{1-\beta}\Big)^{L}~=~(K-1)!! \;\Big(\frac{1}{1-\beta}\Big)^{K/2}\,,
\end{align}
which are the moments $m_{K}\left(\sN\left(0,\eh[1-\beta]\right)\right)$ of a normal distribution with mean zero and variance $\eh[1-\beta]$ for even $K$.

\section{Proof of Theorem \ref{thm:nclt}}\label{sec:nclt}
To prove Theorem \ref{thm:nclt} we have to estimate

\begin{align}
     &\eh[N^{\frac{3}{4}K}]\;\IE_{1,N}
\left(\left(\sum_{i=1}^{N} X_{i}\right)^{K}\right)\\
~=~&\eh[N^{\frac{3}{4}K}]\sum_{r=0}^{K-1}
\sum_{\ub\in W_{K,N}(r)} \IE_{1,N}\,\Big(X_{i_{1}}\cdot X_{i_{2}}\cdot\ldots\cdot X_{i_{K}}\Big)\label{eq:nclt1}\\
~+~&\eh[N^{\frac{3}{4}K}]\sum_{\ub\in W_{K,N}(K)} \IE_{1,N}\,\Big(X_{i_{1}}\cdot X_{i_{2}}\cdot\ldots\cdot X_{i_{K}}\Big)\label{eq:nclt2}\,.
\end{align}
Due to Theorem \ref{thm:Corr} equation \eqref{eq:CorrCW2} and estimate \eqref{eq:comb1} the term \eqref{eq:nclt1} goes to zero. The second term \eqref{eq:nclt2} can be estimated
by Theorem \ref{thm:Corr} equation \eqref{eq:CorrCW2} and \eqref{eq:comb1K}
\begin{align}
  &\eh[N^{\frac{3}{4}K}]\sum_{\ub\in W_{K,N}(K)} \IE_{1,N}\,\Big(X_{i_{1}}\cdot X_{i_{2}}\cdot\ldots\cdot X_{i_{K}}\Big)\nt
  ~\approx~&\eh[N^{\frac{3}{4}K}]\;N^{K}\;\eh[N^{\frac{1}{4}}]\,\frac{\int t^{\ell}\;e^{-\frac{1}{12}t^{4}}\,dt}{\int \;e^{-\frac{1}{12}t^{4}}\,dt}\,.
\end{align}
This gives the result.

\section{Appendix}\label{sec:app}

In this section we give a rough sketch of a proof of Theorem \ref{pro:Laplace}, details to justify the approximations made below can be found in \cite{Olver} or \cite{Kirsch-Mom}.

Without loss of generality we may assume that $t_{0}=0$. To approximate the left hand side of \eqref{eq:Laplace} we make a Taylor expansion $F(t)~\approx~\eh[m!] F^{(m)}(0)\,t^{m}$.
We obtain
\begin{align}
   &\int_{-\infty}^{+\infty} e^{-N\, F(t)}~~ t^{\ell}\, \varphi(t)\, dt~\approx~\int_{-\infty}^{+\infty} e^{-N\, \eh[m!]\,F^{(m)}(0)\,t^{m}}~~ t^{\ell}\, \varphi(t)\, dt\nt
   \intertext{setting $s=(NF^{(m)}(0))^{1/m}\,t$ we get}
   &\approx \eh[(NF^{(m)}(0))^{1/m}] \int_{-\infty}^{+\infty} e^{-\eh[m!]\,s^{m}}~~ \Big(\eh[(NF^{(m)}(0))^{1/m}]s\Big)^{\ell}\, \varphi\Big(\eh[(NF^{(m)}(0))^{1/m}]s\Big)\, ds\nt
   &\approx \eh[(NF^{(m)}(0))^{(\ell+1)/m}]\;\varphi(0)\; \int_{-\infty}^{+\infty} e^{-\eh[m!]\,s^{m}}~~ s^{\ell}\,  ds\,.
\end{align}

\end{document}